\documentclass[11pt]{article}
\usepackage[T1]{fontenc}
\usepackage{lmodern,amsmath,amsthm,amsfonts,amssymb,graphicx,float,microtype,thmtools,underscore,mathtools,anyfontsize} 
\usepackage[usenames,dvipsnames,svgnames,table]{xcolor}
\usepackage[unicode=true]{hyperref}
\hypersetup{%
	colorlinks,
	linkcolor={black},
	breaklinks=true,
	citecolor={black},
	urlcolor={blue!60!black},
	pdftitle={Graphs of Linear Growth},
	pdfauthor={Campbell--Distel--Gollin--Harvey--Hendrey--Hickingbotham--Mohar--Wood}
}
\usepackage[noabbrev,capitalise]{cleveref}
\crefname{lem}{Lemma}{Lemmas}
\crefname{thm}{Theorem}{Theorems}
\crefname{cor}{Corollary}{Corollaries}
\crefname{prop}{Proposition}{Propositions}
\crefname{conj}{Conjecture}{Conjectures}
\crefname{rmk}{Remark}{Remarks}
\crefname{openproblem}{Open Problem}{Open Problems}
\crefformat{equation}{(#2#1#3)}
\Crefformat{equation}{Equation #2(#1)#3}
\usepackage[shortlabels]{enumitem}
\setlist[itemize]{topsep=0ex,itemsep=0ex,parsep=0.4ex}
\setlist[enumerate]{topsep=0ex,itemsep=0ex,parsep=0.4ex}
\newcommand{\defn}[1]{\textcolor{Maroon}{\emph{#1}}}

\newcommand{\GG}{\mathcal{G}}

\newcommand{\NN}{\mathbb{N}}
\newcommand{\RR}{\mathbb{R}}

\usepackage[longnamesfirst,numbers,sort&compress]{natbib}
\makeatletter
\def\NAT@spacechar{~}
\makeatother
\usepackage[bmargin=25mm,tmargin=25mm,lmargin=35mm,rmargin=35mm]{geometry}
\allowdisplaybreaks

\DeclarePairedDelimiter{\ceil}{\lceil}{\rceil}

\renewcommand{\geq}{\geqslant}
\renewcommand{\leq}{\leqslant}
\DeclareMathOperator{\dist}{dist}

\DeclareMathOperator{\tw}{tw}

\DeclareMathOperator{\sep}{sep}

\renewcommand{\thefootnote}{\fnsymbol{footnote}}
\allowdisplaybreaks
\theoremstyle{plain}
\newtheorem{thm}{Theorem}
\newtheorem{lem}[thm]{Lemma}
\newtheorem{cor}[thm]{Corollary}
\newtheorem{prop}[thm]{Proposition}

\theoremstyle{definition}
\newtheorem{conj}[thm]{Conjecture}

\begin{document}

\title{\bf\fontsize{16pt}{0pt}\selectfont  
Graphs of Linear Growth have Bounded Treewidth}

\author{
Rutger Campbell\,\footnotemark[1] \qquad
Marc Distel\,\footnotemark[2] \qquad
J.~Pascal Gollin\,\footnotemark[1] \\
Daniel~J.~Harvey\, \qquad
Kevin Hendrey\,\footnotemark[1] \qquad
Robert~Hickingbotham\,\footnotemark[2] \\
Bojan Mohar\,\footnotemark[3] \qquad
David~R.~Wood\,\footnotemark[2]
}

\footnotetext[1]{Discrete Mathematics Group, Institute for Basic Science (IBS), Daejeon, Republic~of~Korea (\texttt{\{rutger,pascalgollin,kevinhendrey\}@ibs.re.kr}). Research of R.C.\ and K.H.\ supported by the Institute for Basic Science (IBS-R029-C1). Research of J.P.G.\ supported by the Institute for Basic Science (IBS-R029-Y3).}

\footnotetext[2]{School of Mathematics, Monash University, Melbourne, Australia (\texttt{\{marc.distel,robert.hickingbotham, david.wood\}@monash.edu}). Research of D.W.\ supported by the Australian Research Council. Research of M.D.\ and R.H.\ supported by Australian Government Research Training Program Scholarships.}

\footnotetext[3]{Department of Mathematics, Simon Fraser University, Burnaby, British Columbia, Canada (\texttt{mohar@sfu.ca}). Supported in part by the NSERC grant R611450.}

\sloppy
	
\maketitle

\begin{abstract}
    A graph class~$\GG$ has linear growth if, for each graph~${G \in \GG}$ and every positive integer~$r$, every subgraph of~$G$ with radius at most~$r$ contains~${O(r)}$ vertices. 
    In this paper, we show that every graph class with linear growth has bounded treewidth. 
\end{abstract}


\section{Introduction}
\renewcommand{\thefootnote}{\arabic{footnote}}

The \defn{growth} of a (possibly infinite) graph\footnote{We consider undirected graphs~$G$ with vertex-set~${V(G)}$ and edge-set~${E(G)}$. For integers~${m,n \in \mathbb{Z}}$, let ${[m,n] := \{ z \in \mathbb{Z} \colon m \leq z \leq n\}}$ and~${[n]:=[1,n]}$. Let~$\NN$ be the set of positive integers. Let \defn{$\log$} be the natural logarithm $\log_e$.} 
$G$ is the function~${f_G \colon \NN \to \NN \cup \{\infty\}}$ where~${f_G(r)}$ is the supremum of~${|V(H)|}$ taken over all subgraphs~$H$ of~$G$ with radius at most~$r$. 
Growth in graphs is an important topic in group theory~\citep{Sam02,Mil68,Gro81,Gro87,Gri91,Wolf68,GH97}, where growth of a finitely generated group is defined through the growth of the corresponding Cayley graphs. 
Growth of graphs also appears in metric geometry~\citep{KL07}, algebraic graph theory~\citep{Trofimov84,IS91,GISWW89,GS92,IS87,IS86}, and in models of random infinite planar graphs~\citep{EL21,Angel03}. 
A graph class~$\GG$ has \defn{linear}/\defn{quadratic}/\defn{polynomial}/\defn{exponential growth} if~${\sup \{ f_G(r) \colon G \in \GG \}}$ is bounded from above and below by a linear/quadratic/polynomial/exponential function of~$r$.

This paper focuses on graph classes with linear growth. 
Linear growth has previously been studied in the context of infinite vertex-transitive graphs~\citep{Trofimov84,IS91,GISWW89,GS92,IS87,IS86}. 
Notably, \citet{IS86} characterised when an infinite vertex-transitive graph has linear growth in terms of its automorphism group. 
We take a more structural and less algebraic approach, and prove that graph classes with linear growth have  a tree-like structure. 

To formalise this result, we need the following definition. 
A \defn{tree-decomposition} of a graph~$G$ is a collection~${( B_x \subseteq V(G) \colon x \in V(T) )}$ of subsets of~${V(G)}$ (called \defn{bags}) indexed by the nodes of a tree~$T$, such that:  
\begin{itemize}
    \item for every edge~${uv \in E(G)}$, some bag~$B_x$ contains both~$u$ and~$v$, and
    \item for every vertex~${v \in V(G)}$, the set~${\{ x \in V(T) \colon v \in B_x \}}$ induces a non-empty subtree of $T$.
\end{itemize}  
The \defn{width} of a tree-decomposition is the size of the largest bag, minus~$1$. 
The \defn{treewidth} of a graph $G$, denoted by \defn{$\tw(G)$}, is the smallest integer~$w$ for which there is a tree-decomposition of~$G$ of width~$w$, or~$\infty$ if no such~$w$ exists. Treewidth can be thought of as measuring how structurally similar a graph is to a tree. 
Indeed, a connected graph has treewidth~$1$ if and only if it is a tree. 
Treewidth is of fundamental importance in structural and algorithmic graph theory; see \citep{Reed03,HW17,Bodlaender98} for surveys.

Our main result shows that graphs with linear growth have bounded treewidth.

\begin{restatable}{thm}{Main}\label{Main}
    For any~${c \geq 1}$, every graph~$G$ with growth~${f_G(r) \leq cr}$ has treewidth at most~${49c^2+30c}$. 
\end{restatable}

It suffices to prove \cref{Main} for finite graphs, since for~${k \in \NN}$, an infinite graph has treewidth at most~$k$ if and only if every finite subgraph has treewidth at most~$k$ (see \citep{Thomas88,Thomassen89}). 

\cref{Main} is proved in \cref{GrowthTreewidth}, where we also prove an~${\Omega(c\log c)}$ lower bound on the treewidth in \cref{Main}. 
\cref{GrowthMinors} considers graphs with linear growth in proper minor-closed graph classes. 
In this case, we improve the upper bound on the treewidth in \cref{Main} to~${O(c)}$. 
\cref{ProductStructure} explores the product structure of graphs with linear growth. 
Combining \cref{Main} with results from the literature, we show that graphs with linear growth are subgraphs of bounded `blow-ups' of trees with bounded maximum degree, which is a qualitative strengthening of \cref{Main}. 
This section also presents two conjectures about the product structure of graphs with linear and polynomial growth. 
Finally, \cref{Subdivisions} studies the growth of subdivisions of graphs. 
We show that a finite graph with bounded treewidth and bounded maximum degree has a subdivision with linear growth. 
We also show that for any superlinear function~$f$ with~${f(r) \geq 1+\Delta r}$, every finite graph with maximum degree~$\Delta$ (regardless of its treewidth) has a subdivision with growth bounded above by~$f$. 
These results show that, for instance, in \cref{Main}, ``treewidth'' cannot be replaced by ``pathwidth'', while ``${cr}$'' cannot be replaced by ``${cr^{1+\varepsilon}}$''. 

Graphs with bounded treewidth have many attractive properties, and \cref{Main} implies that all such properties hold for graphs of linear growth. 
To conclude this introduction, we give one such example. 
A \defn{$k$-stack layout} of a graph~$G$ is a pair~${(\leq,\varphi)}$ where $\leq$ is a linear ordering on $V(G)$ and ${\varphi\colon E(G)\to[k]}$ is a function such that ${\varphi(ux) \neq \varphi(vy)}$ for any two edges~${ux,vy \in E(G)}$ with~${u<v<x<y}$. 
The \defn{stack-number} of a (possibly infinite) graph $G$ is the minimum integer~${k\geq 0}$ such that there exists a $k$-stack layout of $G$, or~$\infty$ if no such~$k$ exists. 
This topic is widely studied; see \citep{DujWoo07,DEHMW22,BBKR17,MBKPRU20,Yann89,Yann20} for example.

\citet{EHMNSW} recently showed that~${P_n \boxtimes P_n \boxtimes P_n}$, which has growth~${(2r+1)^3}$, has unbounded stack-number (as~${n \to \infty}$). 
Motivated by this discovery, they asked whether graphs of quadratic or of linear growth have bounded stack-number. 
\citet{GH-DAM01} showed that every finite graph with treewidth~$k$ has stack-number at most~${k+1}$. 
\cref{Main} thus implies a positive answer to the second part of this question. 

\begin{thm}\label{Stack}
    For any~${c \geq 1}$, every graph~$G$ with growth~${f_G(r) \leq cr}$ has stack-number at most~${49c^2+30c+1}$. 
\end{thm}
 
As before, it suffices to prove \cref{Stack} for finite graphs, since a standard compactness argument shows that for~${k \in \NN}$, an infinite graph has stack-number at most~$k$ if and only if every finite subgraph has stack-number at most~$k$ (see \cref{StackInfinite}). 

For the remainder of the paper, we assume that every graph is finite.

\section{Growth and Treewidth}
\label{GrowthTreewidth}

This section proves our main result, \cref{Main}, as well as a lower bound for the growth of the class of graphs of tree-width  at most $c$, \cref{LowerBound}. 

The key tool we use is that of balanced separations. 
A \defn{separation} of a graph~$G$ is a pair~${(A,B)}$ of subsets of~${V(G)}$ such that~${A \cup B = V(G)}$ and no edge of~$G$ has one end in~${A \setminus B}$ and the other in~${B \setminus A}$. 
The \defn{order} of the separation~${(A,B)}$ is~${|A \cap B|}$. 
For~${\alpha\in [\frac{2}{3},1)}$, a separation~${(A,B)}$ of a graph on~$n$ vertices is \defn{$\alpha$-balanced} if~${|A| \leq \alpha n}$ and~${|B| \leq \alpha n}$. 
The \defn{$\alpha$-separation number} $\sep_{\alpha}(G)$ of a graph~$G$ is the smallest integer~$s$ such that every subgraph of~$G$ has an $\alpha$-balanced separation of order at most~$s$.

\citet{RS-II} showed that~${\sep_{2/3}(G) \leq \tw(G)+1}$. 
\citet{DN19} established the following converse. 

\begin{thm}[\cite{DN19}]\label{BalancedSep}
    For every graph~$G$, ${\tw(G)\leq 15\sep_{2/3}(G)}$.
\end{thm}

The next two lemmas are folklore. 
The first one enables us to work in the more general setting of $\alpha$-balanced separation.

\begin{lem}\label{BalancedSepAlpha}
    For every~${\alpha \in [\frac{2}{3},1)}$ and every graph~$G$, ${\sep_{2/3}(G)\leq \lceil \log_{\alpha}(\frac{2}{3}) \rceil \sep_{\alpha}(G)}$.
\end{lem}

\begin{proof}
    Let $H$ be a subgraph of $G$ and let $n:=|V(H)|$. 
    Let~$(A_1,B_1)$ be an $\alpha$-balanced separation of $H$ with order at most $\sep_{\alpha}(G)$. 
    For~${i\in \NN}$, we iteratively construct some $\max \{\frac{2}{3}, \alpha^i \}$-balanced separation~$(A_i,B_i)$ of $H$ with order at most $i\sep_{\alpha}(G)$. 
    If $\max\{|A_i \setminus B_i|,|B_i \setminus A_i|\} \leq \frac{2}{3}n$, then $(A_i,B_i)$ is a $\frac{2}{3}$-balanced separation of $H$ and we set~$(A_{i+1},B_{i+1}) := (A_i,B_i)$. 
    Otherwise, we may assume that~$|B_i \setminus A_i|>\frac{2}{3}n$. 
    Let $(C_i,D_i)$ be an $\alpha$-balanced separation of $H[B_i \setminus A_i]$ with order at most $\sep_{\alpha}(G)$. 
    Without loss of generality, assume that~${|D_i| \geq |C_i|}$ and hence~$|D_i| > \frac{n}{3}$. 
    Set $A_{i+1}:=A_i\cup C_i$ and $B_{i+1}:=D_i \cup (A_{i} \cap B_{i})$. 
    Thus $|B_{i+1}\setminus A_{i+1}|\leq \alpha|B_i \setminus A_i|\leq \alpha^{i+1}n$ and $|A_{i+1} \cap B_{i+1}|\leq  |A_i \cap B_i|+\sep_{\alpha}(G)\leq (i+1)\sep_{\alpha}(G)$ and $|A_{i+1} \setminus B_{i+1}|\leq n-|D_i|< \frac{2}{3}n$, so~$(A_{i+1},B_{i+1})$ is as desired. 
    
    Now $(A_i,B_i)$ with $i=\lceil \log_{\alpha}(\frac{2}{3})\rceil$ is a $\frac{2}{3}$-balanced separation of $H$ of order at most $i\sep_{\alpha}(G)$, as required. 
\end{proof}

\begin{lem}\label{SeparationConnected}
     For every~${\alpha \in [\frac{2}{3},1)}$ and every graph~$G$, if every connected subgraph of~$G$ has an $\alpha$-balanced separation of order less than~$c$, then~${\sep_{\alpha}(G)< c}$.
\end{lem}

\begin{proof}
    Consider a subgraph $H$ of $G$. Let $n:=|V(H)|$. We prove that $H$ has an $\alpha$-balanced separation of order less than $c$ by induction on the number of components of $H$. If $H$ is connected, then the claim holds by assumption. So assume that $H$ has at least two components and let $J$ be the smallest component of $H$. If $|V(H)\setminus V(J)|\leq \frac{2}{3}n$, then $(V(H)\setminus V(J),V(J))$ is an $\alpha$-balanced separation of $H$ of order $0$. So assume that $|V(H)\setminus V(J)|\geq \frac{2}{3}n$. By induction, $H-V(J)$ has an $\alpha$-balanced separation $(A,B)$ of order less than $c$ such that $|A|\geq \frac{n}{3}$. Therefore, since $|(B\cup V(J))\setminus A|\leq \frac{2n}{3}\leq \alpha n$ and $A\cap V(J)=\emptyset$, it follows that $(A,B\cup V(J))$ is an $\alpha$-balanced separation of $H$ of order less than $c$, as required.
\end{proof}
 
The next lemma is the heart of this paper. 

\begin{lem}\label{LinearBalancedSep}
    For~${c \geq 1}$, every graph~$G$ with growth~${f_G(r) \leq cr}$ satisfies 
    \[
        \sep_{(1-\frac{1}{4c})}(G)< 2c.
    \]
\end{lem}

\begin{proof}
    Consider a connected subgraph $H$ of $G$ and note that $f_H(r)\leq f_G(r)\leq cr$. Let $n:=|V(H)|$. Let $v\in V(H)$, let $p:=\max\{\dist_H(v,w):w\in V(H)\}$, 
    and let $V_i:=\{w\in V(H):\dist_H(v,w)=i\}$ for $i\in[0,p]$.  
    Let $R:=\{i\in [p]: |V_i| \geq 2c\}$ and $S:=\{i\in [p]:|V_i| < 2c\}$. 
    Since $H$ has radius at most $p$, 
    \[
    2c|R|\leq \sum_{i\in R}|V_i|\leq n \leq cp,
    \]
    Therefore  $|R| \leq\frac{p}{2}$ and $|S| = p - |R| \geq \frac{p}{2}$. Let $j$ be the minimum element of $S$ such that $|S\cap [0,j]|\geq \frac{|S|}{2}$. Let $A := \bigcup_{i\in [0,j]}V_i$ and $B := \bigcup_{i\in [j,p]}V_i$. Then $|A \cap B|=|V_j| < 2c$ and
    \[
        |A| 
        \geq \frac{|S|}{2}\geq\frac{p}{4}\geq \frac{n}{4c} 
        \quad \text{and} \quad
        |B| 
        \geq \frac{|S|}{2}\geq \frac{p}{4}\geq \frac{n}{4c}.
   \]
   Since $V_j$ separates $A\setminus B$ and $B\setminus A$, there is no edge of $H$ with one end in $A \setminus B$ and the other in $B \setminus A$. 
   Moreover, since $|A| \geq \frac{n}{4c}$ and $|B| \geq \frac{n}{4c}$, it follows that $|A\setminus B|\leq (1- \frac{1}{4c})n$ and $|B\setminus A|\leq (1- \frac{1}{4c})n$. 
   Thus $(A,B)$ is a $(1- \frac{1}{4c})$-balanced separation of $H$ of order less than~$2c$. Since $1-\frac{1}{4c}\geq \frac{2}{3}$, the result follows by \cref{SeparationConnected}.
\end{proof}

We are now ready to prove our main theorem which we restate for convenience.

\Main*

\begin{proof}
    Let~$G$ be a graph with growth~${f_G(r) \leq cr}$. 
    By \cref{BalancedSepAlpha,LinearBalancedSep}, 
    \[\sep_{2/3}(G)\leq \left\lceil \log_{(1-\frac{1}{4c})}\left(\tfrac{2}{3}\right)\right\rceil \sep_{(1-\frac{1}{4c})}(G)\leq \left\lceil \log_{(1-\frac{1}{4c})}\left(\tfrac{2}{3}\right)\right\rceil 2c.\]
    Note that $\log_{(1-\frac{1}{4c})}(\tfrac{2}{3})=\frac{\log(\frac{3}{2})}{\log(4c)-\log(4c-1)}$. 
    Additionally, 
    by the mean value theorem there is some $x\in (4c-1,4c)$ such that $x^{-1}=\log(4c)-\log(4c-1)$. 
    Combining these observations with \cref{BalancedSep} yields
    \[
        \tw(G)  \leq 15\sep_{2/3}(G) \leq 30\left\lceil \log\left(\tfrac{3}{2}\right)x\right\rceil c \leq 30\left(\log\left(\tfrac{3}{2}\right)4c+1\right)c\leq 49 c^2+30c.\qedhere
    \]
\end{proof}

We conclude this section by showing that the function $49c^2+30c$ in \cref{Main} cannot be replaced by any function in~${o(c\log c)}$. 
For a vertex~$v$ in a graph~$G$, the \defn{$r$-ball} at~$v$ is the set \defn{$B_r(v)$}${\;:=\{w\in V(G) \colon \dist_G(v,w)\leq r\}}$.

\begin{lem}\label{NonLinearGrowth}
There is an absolute constant $\beta>0$ such that, for every integer $k\geq 2$, there is a cubic graph $G$ with treewidth at least $k$ and growth $f_G(r)\leq \frac{\beta kr}{\log k}$. 
\end{lem}

\begin{proof}
\citet{GM09} proved there is an absolute constant $\alpha\in (0,1)$ such that for every even integer $n\geq 4$ there is an $n$-vertex cubic graph with treewidth at least $\alpha n$. Apply this result with $n:= \max\{2 \ceil{\frac{k}{2\alpha}},4\}$ to obtain a cubic graph $G$ with treewidth at least $k$. Let $v\in V(G)$ and $r\in \NN$, and consider the ball $B_r(v)$. Since $G$ is cubic, $\frac{|B_r(v)|}{r}\leq \min\{\frac{n}{r},3\cdot\frac{2^r}{r}\}$, which is maximised when $n=3\cdot2^r$. Thus  $\frac{|B_r(v)|}{r}\leq \frac{n}{\log_2(n/3)}\leq \frac{\beta k}{\log k}$, for some absolute constant $\beta$, as required.
\end{proof}

\begin{restatable}{thm}{LowerBound}\label{LowerBound}
If $g$ is any function such that for any $c\geq 1$, every graph $G$ of growth $f_G(r)\leq cr$ has treewidth at most $g(c)$, then $g(c)\in\Omega(c\log c)$.
\end{restatable}

\begin{proof}
By \cref{NonLinearGrowth}, there is an absolute constant $\beta>0$ such that for every $k\in \NN$ there is a cubic graph $G$ with treewidth at least $k$ and growth $f_G(r)\leq \frac{\beta kr}{\log k}$. 
Let $k$ be sufficiently large so that $\log k \geq \beta$. 
Let $c:= \frac{\beta k}{\log k}$.
It follows that $k\beta \geq c\log c$ and $f_G(r)\leq cr$. Hence $\frac{c\log c}{\beta} \leq k\leq \tw(G) \leq g(c)$, and $g(c)\in\Omega(c\log c)$, as desired. 
\end{proof}

\section{Growth and Minors}
\label{GrowthMinors}

This section studies growth in proper minor-closed graph classes. A graph $H$ is a \defn{minor} of a graph $G$ if $H$ is isomorphic to a graph obtained from a subgraph of $G$ by contracting edges. A graph class $\GG$ is \defn{minor-closed} if for every $G\in\GG$ every minor of $G$ is also in $\GG$. A minor-closed class $\GG$ is \defn{proper} if some graph is not in $\GG$. A graph parameter $\lambda$ is \defn{minor-monotone} if $\lambda(H)\leq\lambda(G)$ whenever $H$ is a minor of $G$. 

Grid graphs are the key examples here. For $n\in\NN$, the \defn{$n \times n$ grid} is the graph with vertex set $\{(v_1,v_2):v_1,v_2\in[n]\}$ where $(v_1,v_2)$ and $(u_1,u_2)$ are adjacent if $v_1=u_1$ and $|v_2-u_2|=1$, or $v_2=u_2$ and $|v_1-u_1|=1$. This graph has treewidth $n$ (see \citep{HW17}), and is a canonical example of a graph with large treewidth in the sense that every graph $G$ with sufficiently large treewidth contains the $n\times n$ grid as a minor~\citep{RS-V}. Since treewidth is minor-monotone, \cref{Main} implies that any graph $G$ with growth $f_G(r)\leq cr$ cannot contain the $n\times n$ grid as a minor, where $n=\lceil 49c^2+30c\rceil$. We prove this directly with $n=\ceil{2c}$.
 
\begin{restatable}{thm}{GrowthGridMinor}\label{thm:GrowthGridMinor}
    For any~${c \geq 1}$, every graph~$G$ with growth~${f_G(r) \leq cr}$ does not contain the $\ceil{2c}\times \ceil{2c}$ grid as a minor.
\end{restatable}
 
\begin{proof}
It is sufficient to consider the case when $2c\in \NN$. 
Suppose for contradiction that $G$ is a graph with growth $f_G(r)\leq cr$ that contains a $2c\times 2c$ grid as a minor. Thus, there is a collection $\mathcal{H}:=\{H_{i,j}:(i,j)\in [2c]^2\}$ of pairwise vertex-disjoint connected subgraphs of $G$ such that for every $i\in [2c]$ and $j\in [2c-1]$ there is an edge between $H_{i,j}$ and $H_{i,j+1}$ and an edge between $H_{j,i}$ and $H_{j+1,i}$. 
For each $i\in [2c]$, let $R_i:=\bigcup_{j\in [2c]}V(H_{i,j})$ and $C_i:=\bigcup_{j\in [2c]}V(H_{j,i})$. Without loss of generality, there exists $x\in [2c]$ such that $s:=|R_x|\leq |C_i|$ for all $i\in [2c]$. Let $v$ be a vertex in $R_x$.

We claim that $|B_{2s-1}(v)\cap C_i|\geq s$ for each $i\in[2c]$. Since $G[R_x]$ is connected, $R_x\subseteq B_{s-1}(v)$, so $B_{s-1}(v)$ contains a vertex of $C_i$. If $C_i\subseteq B_{2s-1}(v)$ then $|B_{2s-1}(v)\cap C_i|= |C_i|\geq |R_x|=s$, as claimed. Otherwise, since $C_i$ is connected, $C_i$ intersects $B_j$ for each $j\in[s-1,2s]$, implying 
$|B_{2s-1}(v)\cap C_i|\geq s$, which proves the claim. 
Since $C_i$ is disjoint from $C_{i'}$ for all distinct $i,i'\in [2c]$, we find that $2cs\leq |B_{2s-1}(v)|\leq c(2s-1)$, which is the desired contradiction.
\end{proof}

\citet{DH08} showed that for any fixed graph~$H$, every $H$-minor-free graph~$G$ with treewidth~$k$ contains an~${\Omega(k) \times \Omega(k)}$ grid as a minor (see \citep{KK20} for explicit bounds). 
In this case, \cref{thm:GrowthGridMinor} implies the following improvement on \cref{Main}.
 
\begin{cor}\label{HMinorFreeGrowth}
    For any~${c \geq 1}$ and any fixed graph~$H$, every $H$-minor-free graph~$G$ with growth~${f_G(r) \leq cr}$ has treewidth at most~${O(c)}$. 
    \qed
\end{cor}
 
In the case of planar graphs, \citet{RST-JCTB94} showed that every planar graph containing no~${n \times n}$-grid minor has treewidth at most $6n-5$. 
\cref{thm:GrowthGridMinor} thus implies:

\begin{cor}
    \label{PlanarGrowth}
    For any~${c \in \NN}$, every planar graph~$G$ with growth~${f_G(r) \leq cr}$ has treewidth at most~${12c+1}$. 
    \qed
\end{cor}

Recall that \cref{LowerBound} provides an~${\Omega(c \log c)}$ lower bound on the treewidth of graphs~$G$ with growth~${f_G(r) \leq cr}$. 
Thus to conclude the~${O(c)}$ upper bounds in \cref{HMinorFreeGrowth,PlanarGrowth}, it is essential to make some assumption such as excluding a fixed minor.

\section{Product Structure}
\label{ProductStructure}

Much of the research on the growth of finite graphs has centred around polynomial growth. In this setting, \citet{KL07} showed that every graph $G$ of growth $f_G(r)\leq r^d$ (for $r\geq 2$) is isomorphic to a subgraph of the strong product of $O(d\log d)$ sufficiently long paths. Here  the \defn{strong product} $ G \boxtimes H $ of  graphs $ G $ and $ H $ 
is the graph with vertex-set $ V(G) \times V(H) $ with an edge between two vertices $ (v,w) $ and $ (v',w') $ if  $ vv' \in E(G) $ and $ w=w'$ or $ ww' \in E(H) $, or $ v=v' $ and $ww'\in E(H)$. Note that $G\boxtimes K_t$ is simply the graph obtained from $G$ by replacing each vertex of $G$ by a copy of $K_t$ and replacing each edge of $G$ by $K_{t,t}$ between the corresponding copies of $K_t$, sometimes called a \defn{blow-up} of $G$. The following result, due to a referee of \cite{DO95} and refined in \citep{DW,Wood09}, allows us to describe graphs of linear growth as subgraphs of blow-ups of trees. 

\begin{lem}[\cite{DO95,Wood09,DW}]
\label{TreeProduct}
For $k,\Delta\in\NN$, any graph with treewidth less than $k$ and maximum degree $\Delta$ is isomorphic to a subgraph of\/ $T\boxtimes K_{18k\Delta}$ for some tree $T$.
\end{lem}

A graph $G$ with growth $f_G(r)\leq cr$ has maximum degree at most $c-1$. Thus the following result\footnotemark\ is a consequence of \cref{Main,TreeProduct}. 

\begin{thm}\label{ProductStructureLinear}
For any~${c \geq 1}$, every graph~$G$ with growth~${f_G(r) \leq cr}$ is isomorphic to a subgraph of\/~${T \boxtimes K_{\lfloor 882c^3 \rfloor}}$ for some tree~$T$.
\end{thm}

\footnotetext{It follows from a result of \citet{KuskeLohrey05} (or \citet{HMSTW} in the countable case) that \cref{ProductStructureLinear} also holds for  infinite graphs.} 

The graph $T\boxtimes K_{\lfloor 882c^3\rfloor}$ preserves the boundedness of the treewidth of $G$. However, the growth of $T\boxtimes K_{\lfloor 882c^3\rfloor}$ is at least the growth of $T$ which can be exponential, for example if $T$ is a complete binary tree. This leads us to conjecture the following rough characterisation of graphs of linear growth. 

\begin{conj}
\label{LinearGrowthConj}
There exist functions~${g \colon \RR \to \NN}$ and~${h \colon \RR \to \RR}$ such that for any~${c \geq 1}$, every graph $G$ with growth $f_G(r)\leq cr$ is isomorphic to a subgraph of $T\boxtimes K_{g(c)}$ for some tree $T$ with growth $f_T(r)\leq h(c)r$.
\end{conj}

This conjecture (if true) would approximately characterise graphs of linear growth in the sense that every subgraph $H$ of $T\boxtimes K_{g(c)}$ has growth $f_H(r)\leq g(c)h(c)r\in O(r)$.

More generally, for graphs of polynomial growth, we conjecture the following rough characterisation.

\begin{conj}
\label{PolyGrowth}
There exist functions~${g \colon \RR \times \NN \to \NN}$ and~${h \colon \RR \times \NN \to \RR}$ such that for any~${c \geq 1}$ and~${d \in \NN}$, every graph~$G$ with growth $f_G(r)\leq cr^d$ is isomorphic to a subgraph of $T_1\boxtimes\cdots\boxtimes T_d\boxtimes K_{g(c,d)}$, where each $T_i$ is a tree of growth $f_{T_i}(r)\leq h(c,d)r$.
\end{conj}


Again, this conjecture (if true) would approximately characterise graphs of degree-$d$ polynomial growth in the sense that if $H$ is a subgraph of $T_1\boxtimes\cdots\boxtimes T_d\boxtimes K_{g(c,d)}$, then $f_H(r) \leqslant g(c,d) \, ( h(c,d)r)^d \in O(r^d)$.

\section{Growth and Subdivisions}
\label{Subdivisions}

This section considers the growth of subdivisions of graphs. A graph~$\tilde{G}$ is a \defn{subdivision} of a graph~$G$ if~$\tilde{G}$ can be obtained from~$G$ by replacing each edge~${vw}$ by a path~$P_{vw}$ with endpoints~$v$ and~$w$ (internally disjoint from the rest of~$\tilde{G}$).
If each of these paths has the same length, then $\tilde{G}$ is said to be \defn{uniform}.

\cref{Subdivisiontw} below shows that every graph with bounded degree and bounded treewidth has a subdivision with linear growth. 
By \cref{TreeProduct}, we can obtain \cref{Subdivisiontw} from the following result.


\begin{lem}\label{Subdivisiontpw}
For any $k,\Delta\in\NN$, any $\varepsilon>0$, any tree $T$, and any subgraph $G$ of $T\boxtimes K_k$ with maximum degree at most $\Delta$, there is a subdivision $\tilde{G}$ of $G$ with growth $f_{\tilde{G}}(r)\leq (k\Delta+\varepsilon)r+1$.
\end{lem}

\begin{proof}
Let $n:=|V(T)|$, let $V(T)=\{v_{i}:i\in [0,n-1]\}$, and let $V(K_k)=\{w_i:i\in [k]\}$. For each edge $e=(v_a,w_b)(v_c,w_d)\in E(G)$, let 
$\gamma(e):=\min\{\dist_T(v_0,v_a),\dist_T(v_0,v_c)\}$. 
For every $i\in [0,n-1]$, let $\ell(i)$ be the number of edges $e$ of $G$ with $\gamma(e)> i$. 
Let $g:\NN\to \NN$ be a function such that $g(n)=1$ and $\varepsilon g(r)\geq 2g(r+1)\ell(r)+|V(G)|$ for all $r\in [0,n-1]$.
Let $\tilde{G}$ be the subdivision of $G$ obtained by replacing each edge $e\in E(G)$ by a path of length $2g(\gamma(e))$.

For a vertex $v\in V(\tilde{G})$ and a positive integer $r$, consider the ball $B_r(v)$ in $\tilde{G}$.
If there is no edge $xy\in E(G)$ such that $x,y\in B_r(v)\setminus\{v\}$, then $G[B_r(v)]$ is a subdivision of a star and $|B_r(v)|\leq 1+\Delta r$, as required.

Otherwise, let $h:=\min\{\gamma(xy):xy\in E(G), x,y\in B_r(v)\setminus\{v\} \}$. Note that $r\geq g(h)$. 
Let $S_1$ be the set of subdivision vertices of edges $e\in E(G)$ with $\gamma(e)\leq h$, and let $S_2$ be the set of subdivision vertices of edges $e\in E(G)$ with $\gamma(e)>h$.
By the definition of $g$, we have that $|S_2|+|V(G)|\leq \varepsilon r$.
Since $\tilde{G}[B_r(v)]$ is connected and by the definition of $h$, there is no vertex $(v_i,w_j)\in V(G)\cap B_r(v)$ such that $v_i$ is at distance less than $h$ from $v_0$ in $T$.
Hence $B_r(v)$ contains subdivision vertices of at most $k\Delta$ edges $e$ of $G$ with $\gamma(e)\leq h$.

Suppose that $r< 2g(h)$, and note that $v$ is a subdivision vertex of some edge $x_0y_0\in E(G)$ with $\gamma(x_0y_0)\leq h$.
Consider an edge $xy\in E(G)\setminus \{x_0y_0\}$ such that $\gamma(e)\leq h$. Note that the total number of subdivision vertices of $xy$ or $x_0y_0$ contained in $B_r(v)$ is at most $2r$, and that $B_r(v)$ contains at least as many subdivision vertices of $x_0y_0$ as of $xy$.
It follows that $|B_r(v)\cap S_1|\leq k\Delta r$.

Now suppose that $r\geq 2g(h)$. In this case, every edge $e\in E(G)$ with $\gamma(e)=h$ has fewer than $r$ subdivision vertices and $B_r(v)$ contains at most $r$ subdivision vertices of each edge $e\in E(G)$ with $\gamma(e)<h$, and so again $|B_r(v)\cap S_1|\leq k\Delta r$.
Hence in both cases $|B_r(v)|\leq (k\Delta+\varepsilon)r$, as required. 
\end{proof}

The following theorem is a direct consequence of \cref{TreeProduct,Subdivisiontpw}.

\begin{thm}
    \label{Subdivisiontw}
    For any~${k, \Delta \in \NN}$ and~${\varepsilon > 0}$, every graph~$G$ with maximum degree~$\Delta$ and treewidth less than~$k$ has a subdivision~$\tilde{G}$ with growth~${f_{\tilde{G}}(r) \leq (18k\Delta^2+\varepsilon)r+1}$. \qed
\end{thm}

Note that the growth of any subdivision~$\tilde{G}$ of a graph~$G$ depends on the maximum degree~$\Delta$ of~$G$ (since~${f_{\tilde{G}}(1) \geq \Delta+1}$). We now show that every graph~$G$ with bounded treewidth is a minor of a graph~$G'$ with linear growth where the growth of~$G'$ does not depend on the maximum degree of~$G$. 
\citet{MS11} proved that every graph~$G$ is a minor of some graph~$G'$ with maximum degree~$3$ and treewidth at most~${\tw(G)+1}$. 
\cref{Subdivisiontw} applied to~$G'$ gives the following result. 

\begin{cor}
\label{MinorTW}
    For every~${k \in \NN}$ and~${\varepsilon > 0}$, every graph~$G$ with treewidth less than~$k$ is a minor of some graph~${\tilde{G}}$ with growth~${f_{\tilde{G}}(r) \leq (162(k+1)+\varepsilon)r+1}$.
    \qed
\end{cor}

\cref{Subdivisiontw} implies that \cref{Main} is best possible in the sense that treewidth cannot be replaced by any parameter that is unbounded for graphs of bounded treewidth and bounded maximum degree, and does not decrease when taking subdivisions. 
As an example, pathwidth\footnote{A \defn{path-decomposition} is a tree-decomposition that is indexed by the nodes of a path. 
The \defn{pathwidth} of a graph~$G$ is the minimum width of a path-decomposition of~$G$.}
is a graph parameter that is unbounded on trees with maximum degree~$3$ and does not decrease when taking subdivisions. In particular, there is a class of trees that has linear growth and unbounded pathwidth. Thus, ``treewidth'' cannot be replaced by ``pathwidth'' in \cref{Main}.

Finally, consider subdividing graphs with bounded maximum degree without the assumption of bounded treewidth. 
While \cref{Main} implies that we cannot obtain linear growth in this more general setting, we can get arbitrarily close in the following sense. 
A function~${f \colon \NN \to \RR}$ is \defn{superlinear} if~${\frac{f(x)}{x} \to \infty}$ as~${x \to \infty}$. 
We now show that for every superlinear function~$f$ with~${f(r)\geq 1+\Delta r}$, every graph with maximum degree at most~$\Delta$ admits a subdivision of growth at most~$f$.

\begin{thm}
    \label{SubdivisionSuperlinear2}
For any~$\Delta\in\NN$ and any superlinear function~${f \colon \NN \to \RR}$ with~${f(r) \geq \Delta r+1}$, every graph~$G$ with maximum degree~$\Delta$ has a uniform subdivision~$\tilde{G}$ with growth~${f_{\tilde{G}}(r) \leq f(r)}$.
\end{thm}

\begin{proof}
    Let~${m := |E(G)|}$, ${n := |V(G)|}$, and let~${\ell \in \NN}$ be such that~${f(r) \geq 2r m+n}$ for all~${r \geq \ell}$ (which exists since~$f$ is superlinear). 
    Let~$\tilde{G}$ be obtained from~$G$ by subdividing every edge~$2\ell$ times. 
    Now consider~${r \in \NN}$ and a vertex~${v \in V(\tilde{G})}$. 
    If~${r \geq \ell}$, then ${|B_r(v)| \leq |V(\tilde{G})| \leq 2\ell m+n\leq f(r)}$. 
    If~${r\leq \ell}$, then~${\tilde{G}[B_r(v)]}$ is isomorphic to a subdivision of a star, so~${|B_r(v)|\leq 1+\Delta r
    \leq f(r)}$, as required.
\end{proof}

As mentioned above, every graph is a minor of some graph of maximum degree $3$. Hence we have the following immediate corollary of \cref{SubdivisionSuperlinear2}. 

\begin{cor}\label{MinorSuperlinear}
For any superlinear function     $f:\NN\to\RR$ with $f(r)\geq 1+3r$, every graph $G$ is a minor of a graph $\tilde{G}$ with growth $f_{\tilde{G}}(r)\leq f(r)$.
\end{cor}

\cref{MinorSuperlinear} shows that, for any superlinear function $f$ with $f(r)\geq 1+3 r$, any minor-monotone graph parameter that is unbounded on the class of all graphs is also unbounded on the class of graphs $G$ with $f_G(r)\leq f(r)$.
For example, ``$cr$'' cannot be replaced by ``$cr^{1+\varepsilon}$'' in \cref{Main}. 

\renewcommand{\thefootnote}{\arabic{footnote}}

\subsection*{Acknowledgements} 

This research was initiated at the
\href{https://www.matrix-inst.org.au/events/structural-graph-theory-downunder-ll/}{Structural Graph Theory Downunder II} 
program of the Mathematical Research Institute MATRIX (March 2022). 
 
{
\fontsize{10pt}{11pt}
\selectfont
\bibliographystyle{DavidNatbibStyle}
\bibliography{main.bbl}

\def\soft#1{\leavevmode\setbox0=\hbox{h}\dimen7=\ht0\advance \dimen7
  by-1ex\relax\if t#1\relax\rlap{\raise.6\dimen7
  \hbox{\kern.3ex\char'47}}#1\relax\else\if T#1\relax
  \rlap{\raise.5\dimen7\hbox{\kern1.3ex\char'47}}#1\relax \else\if
  d#1\relax\rlap{\raise.5\dimen7\hbox{\kern.9ex \char'47}}#1\relax\else\if
  D#1\relax\rlap{\raise.5\dimen7 \hbox{\kern1.4ex\char'47}}#1\relax\else\if
  l#1\relax \rlap{\raise.5\dimen7\hbox{\kern.4ex\char'47}}#1\relax \else\if
  L#1\relax\rlap{\raise.5\dimen7\hbox{\kern.7ex
  \char'47}}#1\relax\else\message{accent \string\soft \space #1 not
  defined!}#1\relax\fi\fi\fi\fi\fi\fi}
\begin{thebibliography}{43}
\providecommand{\natexlab}[1]{#1}
\providecommand{\msn}[1]{MR:\,\href{http://www.ams.org/mathscinet-getitem?mr=MR{#1}}{#1}}
\providecommand{\ZBL}[1]{Zbl:\,\href{https://www.zentralblatt-math.org/zmath/en/search/?q=an:#1}{#1}}
\providecommand{\url}[1]{\texttt{#1}}
\providecommand{\urlprefix}{}
\expandafter\ifx\csname urlstyle\endcsname\relax
  \providecommand{\doi}[1]{doi:\discretionary{}{}{}#1}\else
  \providecommand{\doi}{doi:\discretionary{}{}{}\begingroup
  \urlstyle{rm}\Url}\fi

\bibitem[{Angel(2003)}]{Angel03}
\textsc{Omer Angel}.
\newblock \href{https://doi.org/10.1007/s00039-003-0436-5}{Growth and
  percolation on the uniform infinite planar triangulation}.
\newblock \emph{Geom. Funct. Anal.}, 13(5):935--974, 2003.

\bibitem[{Bekos et~al.(2017)Bekos, Bruckdorfer, Kaufmann, and
  Raftopoulou}]{BBKR17}
\textsc{Michael~A. Bekos, Till Bruckdorfer, Michael Kaufmann, and Chrysanthi~N.
  Raftopoulou}.
\newblock \href{https://doi.org/10.1007/s00453-016-0203-2}{The book thickness
  of 1-planar graphs is constant}.
\newblock \emph{Algorithmica}, 79(2):444--465, 2017.

\bibitem[{Bodlaender(1998)}]{Bodlaender98}
\textsc{Hans~L. Bodlaender}.
\newblock \href{https://doi.org/10.1016/S0304-3975(97)00228-4}{A partial
  $k$-arboretum of graphs with bounded treewidth}.
\newblock \emph{Theoret. Comput. Sci.}, 209(1-2):1--45, 1998.

\bibitem[{Demaine and Hajiaghayi(2008)}]{DH08}
\textsc{Erik~D. Demaine and MohammadTaghi Hajiaghayi}.
\newblock \href{https://doi.org/10.1007/s00493-008-2140-4}{Linearity of grid
  minors in treewidth with applications through bidimensionality}.
\newblock \emph{Combinatorica}, 28(1):19--36, 2008.

\bibitem[{Diestel(2018)}]{Diestel5}
\textsc{Reinhard Diestel}.
\newblock Graph theory, vol. 173 of \emph{Graduate Texts in Mathematics}.
\newblock Springer, 5th edn., 2018.

\bibitem[{Ding and Oporowski(1995)}]{DO95}
\textsc{Guoli Ding and Bogdan Oporowski}.
\newblock \href{https://doi.org/10.1002/jgt.3190200412}{Some results on tree
  decomposition of graphs}.
\newblock \emph{J. Graph Theory}, 20(4):481--499, 1995.

\bibitem[{Distel and Wood(2022)}]{DW}
\textsc{Marc Distel and David~R. Wood}.
\newblock \href{http://arxiv.org/abs/2210.12577}{Tree-partitions with bounded
  degree trees}.
\newblock 2022, arXiv:2210.12577.

\bibitem[{Dujmovic et~al.(2022)Dujmovic, Eppstein, Hickingbotham, Morin, and
  Wood}]{DEHMW22}
\textsc{Vida Dujmovic, David Eppstein, Robert Hickingbotham, Pat Morin, and
  David~R. Wood}.
\newblock \href{https://doi.org/10.1007/s00493-021-4585-7}{Stack-number is not
  bounded by queue-number}.
\newblock \emph{Combinatorica}, 42(2):151--164, 2022.

\bibitem[{Dujmovi\'{c} and Wood(2007)}]{DujWoo07}
\textsc{Vida Dujmovi\'{c} and David~R. Wood}.
\newblock \href{https://doi.org/10.1007/s00454-007-1318-7}{Graph treewidth and
  geometric thickness parameters}.
\newblock \emph{Discrete Comput. Geom.}, 37(4):641--670, 2007.

\bibitem[{Dvo\v{r}\'{a}k and Norin(2019)}]{DN19}
\textsc{Zden\v{e}k Dvo\v{r}\'{a}k and Sergey Norin}.
\newblock \href{https://doi.org/10.1016/j.jctb.2018.12.007}{Treewidth of graphs
  with balanced separations}.
\newblock \emph{J. Combin. Theory Ser. B}, 137:137--144, 2019.

\bibitem[{Ebrahimnejad and Lee(2021)}]{EL21}
\textsc{Farzam Ebrahimnejad and James~R. Lee}.
\newblock \href{https://doi.org/10.1007/s00440-021-01045-5}{On planar graphs of
  uniform polynomial growth}.
\newblock \emph{Probab. Theory Related Fields}, 180(3-4):955--984, 2021.

\bibitem[{Eppstein et~al.(pear)Eppstein, Hickingbotham, Merker, Norin, Seweryn,
  and Wood}]{EHMNSW}
\textsc{David Eppstein, Robert Hickingbotham, Laura Merker, Sergey Norin,
  Micha{\l}~T. Seweryn, and David~R. Wood}.
\newblock \href{http://arxiv.org/abs/2202.05327}{Three-dimensional graph
  products with unbounded stack-number}.
\newblock \emph{Discrete Comput. Geom.}, to appear.
\newblock arXiv:2202.05327.

\bibitem[{Ganley and Heath(2001)}]{GH-DAM01}
\textsc{Joseph~L. Ganley and Lenwood~S. Heath}.
\newblock \href{https://doi.org/10.1016/S0166-218X(00)00178-5}{The pagenumber
  of $k$-trees is ${O}(k)$}.
\newblock \emph{Discrete Appl. Math.}, 109(3):215--221, 2001.

\bibitem[{Godsil et~al.(1989)Godsil, Imrich, Seifter, Watkins, and
  Woess}]{GISWW89}
\textsc{Chris~D. Godsil, Wilfried Imrich, Norbert Seifter, Mark~E. Watkins, and
  Wolfgang Woess}.
\newblock \href{https://doi.org/10.1007/BF01788688}{A note on bounded
  automorphisms of infinite graphs}.
\newblock \emph{Graphs Combin.}, 5(4):333--338, 1989.

\bibitem[{Godsil and Seifter(1992)}]{GS92}
\textsc{Christopher~D. Godsil and Norbert Seifter}.
\newblock \href{https://doi.org/10.1007/BF02349960}{Graphs with polynomial
  growth are covering graphs}.
\newblock \emph{Graphs Combin.}, 8(3):233--241, 1992.

\bibitem[{Grigorchuk(1991)}]{Gri91}
\textsc{Rostislav~I. Grigorchuk}.
\newblock On growth in group theory.
\newblock In \emph{Proc. {I}nternational {C}ongress of {M}athematicians
  \textup{({K}yoto 1990)}, {V}ol. {I}, {II}}, pp. 325--338. Math. Soc. Japan,
  Tokyo, 1991.

\bibitem[{Grigorchuk and de~la Harpe(1997)}]{GH97}
\textsc{Rostislav~I. Grigorchuk and Pierre de~la Harpe}.
\newblock \href{https://doi.org/10.1007/BF02471762}{On problems related to
  growth, entropy, and spectrum in group theory}.
\newblock \emph{J. Dynam. Control Systems}, 3(1):51--89, 1997.

\bibitem[{Grohe and Marx(2009)}]{GM09}
\textsc{Martin Grohe and D{\'a}niel Marx}.
\newblock \href{https://doi.org/10.1016/j.jctb.2008.06.004}{On tree width,
  bramble size, and expansion}.
\newblock \emph{J. Combin. Theory Ser. B}, 99(1):218--228, 2009.

\bibitem[{Gromov(1981)}]{Gro81}
\textsc{Mikhael Gromov}.
\newblock \href{http://www.numdam.org/item?id=PMIHES_1981__53__53_0}{Groups of
  polynomial growth and expanding maps}.
\newblock \emph{Inst. Hautes \'{E}tudes Sci. Publ. Math.}, 53:53--73, 1981.

\bibitem[{Gromov(1987)}]{Gro87}
\textsc{Mikhael Gromov}.
\newblock \href{https://doi.org/10.1007/978-1-4613-9586-7\_3}{Hyperbolic
  groups}.
\newblock In \emph{Essays in group theory}, vol.~8 of \emph{Math. Sci. Res.
  Inst. Publ.}, pp. 75--263. Springer, 1987.

\bibitem[{Harvey and Wood(2017)}]{HW17}
\textsc{Daniel~J. Harvey and David~R. Wood}.
\newblock \href{https://doi.org/10.1002/jgt.22030}{Parameters tied to
  treewidth}.
\newblock \emph{J. Graph Theory}, 84(4):364--385, 2017.

\bibitem[{Huynh et~al.(2021)Huynh, Mohar, {\v{S}}{\'a}mal, Thomassen, and
  Wood}]{HMSTW}
\textsc{Tony Huynh, Bojan Mohar, Robert {\v{S}}{\'a}mal, Carsten Thomassen, and
  David~R. Wood}.
\newblock \href{http://arxiv.org/abs/2109.00327}{Universality in minor-closed
  graph classes}.
\newblock 2021, arXiv:2109.00327.
\newblock ArXiv:2109.00327.

\bibitem[{Imrich and Seifter(1987)}]{IS87}
\textsc{Wilfried Imrich and Norbert Seifter}.
\newblock \href{https://doi.org/10.1007/BF01189278}{A bound for groups of
  linear growth}.
\newblock \emph{Arch. Math. (Basel)}, 48(2):100--104, 1987.

\bibitem[{Imrich and Seifter(1989)}]{IS86}
\textsc{Wilfried Imrich and Norbert Seifter}.
\newblock \href{https://doi.org/10.1016/0012-365X(88)90138-0}{A note on the
  growth of transitive graphs}.
\newblock \emph{Discrete Math.}, 73:111--117, 1989.

\bibitem[{Imrich and Seifter(1991)}]{IS91}
\textsc{Wilfried Imrich and Norbert Seifter}.
\newblock \href{https://doi.org/10.1016/0012-365X(91)90332-V}{A survey on
  graphs with polynomial growth}.
\newblock \emph{Discrete Math.}, 95(1-3):101--117, 1991.

\bibitem[{Kaufmann et~al.(2020)Kaufmann, Bekos, Klute, Pupyrev, Raftopoulou,
  and Ueckerdt}]{MBKPRU20}
\textsc{Michael Kaufmann, Michael~A. Bekos, Fabian Klute, Sergey Pupyrev,
  Chrysanthi~N. Raftopoulou, and Torsten Ueckerdt}.
\newblock \href{https://doi.org/10.20382/jocg.v11i1a12}{Four pages are indeed
  necessary for planar graphs}.
\newblock \emph{J. Comput. Geom.}, 11(1):332--353, 2020.

\bibitem[{Kawarabayashi and Kobayashi(2020)}]{KK20}
\textsc{Ken-ichi Kawarabayashi and Yusuke Kobayashi}.
\newblock \href{https://doi.org/10.1016/j.jctb.2019.07.007}{Linear min-max
  relation between the treewidth of an {$H$}-minor-free graph and its largest
  grid minor}.
\newblock \emph{J. Combin. Theory Ser. B}, 141:165--180, 2020.

\bibitem[{Krauthgamer and Lee(2007)}]{KL07}
\textsc{Robert Krauthgamer and James~R. Lee}.
\newblock \href{https://doi.org/10.1007/s00493-007-2183-y}{The intrinsic
  dimensionality of graphs}.
\newblock \emph{Combinatorica}, 27(5):551--585, 2007.

\bibitem[{Kuske and Lohrey(2005)}]{KuskeLohrey05}
\textsc{Dietrich Kuske and Markus Lohrey}.
\newblock \href{https://doi.org/10.1016/j.apal.2004.06.002}{Logical aspects of
  {C}ayley-graphs: the group case}.
\newblock \emph{Ann. Pure Appl. Logic}, 131(1--3):263--286, 2005.

\bibitem[{Markov and Shi(2011)}]{MS11}
\textsc{Igor Markov and Yaoyun Shi}.
\newblock \href{https://doi.org/10.1007/s00453-009-9312-5}{Constant-degree
  graph expansions that preserve treewidth}.
\newblock \emph{Algorithmica}, 59(4):461--470, 2011.

\bibitem[{Milnor(1968)}]{Mil68}
\textsc{John Milnor}.
\newblock \href{http://projecteuclid.org/euclid.jdg/1214428659}{Growth of
  finitely generated solvable groups}.
\newblock \emph{J. Differential Geometry}, 2:447--449, 1968.

\bibitem[{Reed(2003)}]{Reed03}
\textsc{Bruce~A. Reed}.
\newblock \href{https://doi.org/10.1007/0-387-22444-0\_4}{Algorithmic aspects
  of tree width}.
\newblock In \emph{Recent advances in algorithms and combinatorics}, vol.~11,
  pp. 85--107. Springer, 2003.

\bibitem[{Robertson and Seymour(1986{\natexlab{a}})}]{RS-II}
\textsc{Neil Robertson and Paul Seymour}.
\newblock \href{https://doi.org/10.1016/0196-6774(86)90023-4}{Graph minors.
  {II}. {A}lgorithmic aspects of tree-width}.
\newblock \emph{J. Algorithms}, 7(3):309--322, 1986{\natexlab{a}}.

\bibitem[{Robertson and Seymour(1986{\natexlab{b}})}]{RS-V}
\textsc{Neil Robertson and Paul Seymour}.
\newblock \href{https://doi.org/10.1016/0095-8956(86)90030-4}{Graph minors.
  {V}. {E}xcluding a planar graph}.
\newblock \emph{J. Combin. Theory Ser. B}, 41(1):92--114, 1986{\natexlab{b}}.

\bibitem[{Robertson et~al.(1994)Robertson, Seymour, and Thomas}]{RST-JCTB94}
\textsc{Neil Robertson, Paul Seymour, and Robin Thomas}.
\newblock \href{https://doi.org/10.1006/jctb.1994.1073}{Quickly excluding a
  planar graph}.
\newblock \emph{J. Combin. Theory Ser. B}, 62(2):323--348, 1994.

\bibitem[{Sambusetti(2002)}]{Sam02}
\textsc{Andrea Sambusetti}.
\newblock On minimal growth in group theory and {R}iemannian geometry.
\newblock In \emph{Differential geometry, {V}alencia, 2001}, pp. 268--280.
  World Sci. Publ., 2002.

\bibitem[{Thomas(1988)}]{Thomas88}
\textsc{Robin Thomas}.
\newblock \href{https://people.math.gatech.edu/~thomas/PAP/twcpt.pdf}{The
  tree-width compactness theorem for hypergraphs}.
\newblock 1988.

\bibitem[{Thomassen(1989)}]{Thomassen89}
\textsc{Carsten Thomassen}.
\newblock
  \href{https://doi.org/10.1111/j.1749-6632.1989.tb22479.x}{Configurations in
  graphs of large minimum degree, connectivity, or chromatic number}.
\newblock In \emph{{P}roc.\ 3rd {I}nternational {C}onference on Combinatorial
  {M}athematics}, vol. 555 of \emph{Ann. New York Acad. Sci.}, pp. 402--412.
  1989.

\bibitem[{Trofimov(1984)}]{Trofimov84}
\textsc{Vladimir~I. Trofimov}.
\newblock \href{https://doi.org/10.1070/SM1985v051n02ABEH002866}{Graphs with
  polynomial growth}.
\newblock \emph{Mat. Sb. (N.S.)}, 123(165)(3):407--421, 1984.

\bibitem[{Wolf(1968)}]{Wolf68}
\textsc{Joseph~A. Wolf}.
\newblock \href{http://projecteuclid.org/euclid.jdg/1214428658}{Growth of
  finitely generated solvable groups and curvature of {R}iemannian manifolds}.
\newblock \emph{J. Differential Geometry}, 2:421--446, 1968.

\bibitem[{Wood(2009)}]{Wood09}
\textsc{David~R. Wood}.
\newblock \href{https://doi.org/10.1016/j.ejc.2008.11.010}{On
  tree-partition-width}.
\newblock \emph{European J. Combin.}, 30(5):1245--1253, 2009.

\bibitem[{Yannakakis(1989)}]{Yann89}
\textsc{Mihalis Yannakakis}.
\newblock \href{https://doi.org/10.1016/0022-0000(89)90032-9}{Embedding planar
  graphs in four pages}.
\newblock \emph{J. Comput. System Sci.}, 38(1):36--67, 1989.

\bibitem[{Yannakakis(2020)}]{Yann20}
\textsc{Mihalis Yannakakis}.
\newblock \href{https://doi.org/10.1016/j.jctb.2020.05.008}{Planar graphs that
  need four pages}.
\newblock \emph{J. Combin. Theory Ser. B}, 145:241--263, 2020.

\end{thebibliography}
}

\appendix

\section{Stack-number of infinite graphs}
\label{StackInfinite}

The following result is proved via a standard compactness argument.

\begin{prop}
\label{stacknumber-compactness}
For~$k\in\NN$, a graph~$G$ has stack-number at most~$k$ if and only if every finite subgraph of~$G$ has stack-number at most~$k$.
\end{prop}

First, we introduce a version of the compactness principle in combinatorics; see~\cite[Appendix~A]{Diestel5}. A partially ordered set~${(\mathcal{P}, \leq)}$ is \defn{directed} if any two elements have a common upper bound; that is,~for any~${p,q \in \mathcal{P}}$ there exists~${r \in \mathcal{P}}$ with~${p \leq r}$ and~${q \leq r}$.
A \defn{directed inverse system} consists of 
a directed poset~$\mathcal{P}$, 
a family of sets~${( S_p \colon p \in \mathcal{P} )}$, and
for all~${p, q \in \mathcal{P}}$ with~${p < q}$ a map~${g_{q,p} \colon S_q \to S_p}$
such that the maps are \defn{compatible}; that is,~${g_{q, p} \circ g_{r, q} = g_{r, p}}$ for all ${p, q, r \in \mathcal{P}}$ with~${p < q < r}$.
The \defn{inverse limit} of such a directed inverse system is the set 
\[
    \lim \limits_{\longleftarrow}\, ( S_p \colon p \in \mathcal{P}) = \left\{ ( s_p \colon p \in \mathcal{P}) \in \prod \limits_{p \in \mathcal{P}} S_p \, \colon \, g_{q,p}(s_q) = s_p \textnormal{ for all } p,q \in \mathcal{P} \textnormal{ with } p < q \right\}.
\]

\begin{lem}[\mbox{Generalised Infinity Lemma~\cite[Appendix~A]{Diestel5}}]
    \label{gen-inf-lemma}
    The inverse limit of any directed inverse system of non-empty finite sets is non-empty. 
\end{lem} 

\begin{proof}[Proof of~\cref{stacknumber-compactness}]
    For a linear order~$\leq$ on a set~$X$ and a subset~${Y \subseteq X}$, let~$\leq{\upharpoonright}Y$ denote the restriction of~$\leq$ to~$Y$. 
    Similarly for a function~$\varphi$ with domain~$X$ and a subset~${Y \subseteq X}$, let~$\varphi{\upharpoonright}Y$ denote the restriction of~$\varphi$ to~$Y$. 
    
    ($\Longrightarrow$) Clearly~$(\leq{\upharpoonright}V(H),\varphi{\upharpoonright}E(H))$ is a $k$-stack layout for any $k$-stack layout~${(\leq,\varphi)}$ of~$G$ and any subgraph~$H$ of~$G$. 
        
    ($\Longleftarrow$) Let~$\mathcal{P}$ be the set of finite subsets of~$V(G)$ and consider the directed poset~${(\mathcal{P},\subseteq)}$. 
    For every finite set~${X \subseteq V(G)}$, let~$S_X$ be the set of all $k$-stack layouts of~${G[X]}$. 
    For~${Y \subseteq X \in \mathcal{P}}$ and~${(\leq,\varphi) \in S_X}$, let~${g_{X,Y}(\leq,\varphi) := (\leq{\upharpoonright}Y, \varphi{\upharpoonright}E(G[Y]))}$, and note that~${g_{X,Y}(\leq,\varphi) \in S_Y}$. 
    Moreover, note that for~${Z \subseteq Y \subseteq X \in \mathcal{P}}$ and~${(\leq,\varphi) \in S_X}$, 
    \begin{align*}
        g_{Y,Z}(g_{X,Y}(\leq,\varphi)) 
            = ((\leq{\upharpoonright}Y){\upharpoonright}Z, (\varphi{\upharpoonright}E(G[Y])){\upharpoonright}E(G[Z]))) & = (\leq{\upharpoonright}Z, \varphi{\upharpoonright}E(G[Z])) \\
            &= g_{X,Z}(\leq,\varphi).
    \end{align*}
    Hence, we have a directed inverse system of non-empty finite sets. 
    By the Generalised Infinity Lemma, there is an element~${\big( (\leq_X, \varphi_X) \in S_X \, \colon \, X \in \mathcal{P} \big)}$ in the inverse limit. 
    Define a relation~$\leq$ on~$V(G)$ by setting~${v \leq w}$ if~${v \leq_{\{v,w\}} w}$ for~${v,w \in V(G)}$, and define a function~$\varphi$ on~$E(G)$ by setting~${\varphi(vw) := \varphi_{\{v,w\}}(vw)}$ for~${vw \in E(G)}$. 
    By the compatibility of the maps~$g_{X,Y}$, we have that~$\leq$ is a linear order on~$V(G)$ and~${(\leq{\upharpoonright}X,\varphi{\upharpoonright}E(G[X])) \in S_X}$ for all~${X \in \mathcal{P}}$. 
    Now any two edges $ux$ and~$vy$ with~${u < v < x < y}$ are assigned distinct colours since~${(\leq{\upharpoonright}X,\varphi{\upharpoonright}E(G[X])) \in S_X}$ for~${X = \{u,v,x,y\}}$. 
    Hence~$(\leq,\varphi)$ is a~$k$-stack layout of~$G$. 
\end{proof}
\end{document}